\documentclass[11pt,reqno]{amsart}

\usepackage{amsmath,amssymb}
\usepackage{latexsym}
\usepackage{epsfig}
\usepackage{cite}
\usepackage[english]{babel}
\usepackage{booktabs}
\usepackage{graphicx}
\usepackage{tikz}
\usepackage{pgfplots}
\usepackage{eurosym}
\usepackage[inner=4cm,outer=2cm]{geometry}
\usepackage{amsthm}

\usetikzlibrary{arrows,decorations.pathmorphing,backgrounds,positioning,fit,petri,decorations}
\usetikzlibrary{calc,intersections,through,backgrounds,mindmap,patterns,fadings}
\usetikzlibrary{decorations.text}
\usetikzlibrary{decorations.fractals}
\usetikzlibrary{fadings}
\usepackage{pgflibraryshapes}
\usetikzlibrary{lindenmayersystems}
\usetikzlibrary{shadings}
\usetikzlibrary{shadows}
\usetikzlibrary{shapes.geometric}
\usetikzlibrary{shapes.callouts}
\usetikzlibrary{shapes.misc}
\usetikzlibrary{spy}
\usetikzlibrary{topaths}
\usetikzlibrary{datavisualization}
\usetikzlibrary{datavisualization.formats.functions}

\newtheorem{teo}{Theorem}
\newtheorem{rem}{Remark}
\newtheorem{prop}{Proposition}
\newtheorem{cor}{Corollary}
\newtheorem{tdef}{Definition}

\numberwithin{equation}{section}

\numberwithin{equation}{section}

\title{Macdonald formula, Ricci Curvature, and Concentration Locus for classical compact Lie groups}
\author{S.L.~Cacciatori$^{1,2}$ and P.~Ursino$^{3}$}

\address{$^1$ Department of Science and High Technology, Universit\`a dell'Insubria, Via Valleggio 11, IT-22100 Como, Italy}
\address{$^2$ INFN sezione di Milano, via Celoria 16, IT-20133 Milano, Italy}
\address{$^3$ Department of Mathematics and Informatics, Universit\`a degli Studi di Catania, Viale Andrea Doria 6, 95125 Catania, Italy}

\begin{document}

\maketitle
\begin{abstract}

For Classical compact Lie groups, we use Macdonald's formula \cite{Ma} and Ricci curvature for analyzing a ``concentration locus'', which is a tool to detect where a sequence of metric, Borel measurable spaces concentrates its measure.

\end{abstract}

\section*{Introduction}

\noindent

In the last fifty years, the study of concentration of measure phenomenon has become a research field of powerful interest in different areas of Mathematics. Starting from a pioneering research of Levy in the fifties \cite{Le}, Milman's work in the early seventies, followed by Gromov later work \cite{GrMi}, the notion of Levy Family has been used to study concentration phenomenon (for an exhaustive survey on the subject see \cite{Gro99}, \cite{Pe1},\cite{ledoux} or the most recent \cite{Shi} )

This phenomenon is mostly an asymptotic effect that can be view even in finite dimensional spaces.
In \cite{CU21} we introduced the notion of Concentration Locus, a some kind of ``localized" version of concentration. In a sense, we detect in which part of the spaces, along the process of concentration, the measure concentrates. In the present paper,
we will provide explicit examples showing how such sort of localization takes place for compact Lie groups. We will show different techniques apt to do it:
we use a formula due to Macdonald \cite{Ma}, for computing the volumes of compact simple groups and their subgroups, combined with techniques of realizing group parametrizations, first developed in \cite{CaCeDVOrSc} to infer concentration properties
and to calculate explicitly a concentration locus for some classical lie groups. We also compute the Ricci curvature of such groups and apply Gromov's theorem to deduce Levy property for them. The explicit determination of a concentration locus for the classical
sequence of Lie groups will also show explicitly that it cannot be unique. \\
Finally, we extend our results to families of arbitrary compact connected Lie groups.


\section{Background and statements}
\subsection{Levy Family and Concentration Locus}

\begin{tdef}
For a set $A$ in a metric space $X$ we denote by $N^{\varepsilon}(A)$, $\varepsilon >0$, its e-neighborhood. Consider a family $(X_n,\mu_n)$ with $n:1,2, \dots$ of metric spaces $X_n$ with normalized borel measures $\mu_n$. We call such a family Levy if for any sequence of Borel sets $A_n\subset X_n$  $n:1,2, \dots$, such that $\liminf_{n\rightarrow \infty}\ \mu_n(A_n)>0$, and for every $\varepsilon >0$, we have
$\lim_{n\rightarrow \infty}\ \mu_n(N^{\varepsilon}(A))=1$.
\end{tdef}
\begin{tdef}\label{def2}
 Let $\{X_n, \mu_n\}_{n\in \mathbb N}$ be a family of connected metric spaces with metrics $g_n$, and $\mu_n$ measures w.r.t. which open set are measurable of non-vanishing measure. Assume the measures to be normalized, $\mu_n(X_n)=1$.
 Let $\{S_n\}_{n\in \mathbb N}$ be a family of proper closed subsets, $S_n\subset X_n$. Fix a sequence $\{\varepsilon_n\}_{n\in \mathbb N}$ such that $\varepsilon_n>0$,  $\lim_{n\to\infty}\varepsilon_n=0$, and let
 $\{U^{\varepsilon_n}_n\}_{n\in \mathbb N}$ be the sequence of tubular neighbourhoods of $S_n$ of radius $\varepsilon_n$. We say that the measure concentrates on the family
 $\{S_n\}$ at least at rate of $\varepsilon_n$ if
 \begin{align}
 \lim_{n\to\infty} \mu_n(X_n-U^{\varepsilon_n}_n)=0.
 \end{align}
 We will shortly say that the measure concentrates on $S_n$ and will call it metric concentration. In particular, when $X_n$ are manifolds, we call $S_n$ a concentration locus if it is contained in a submanifold of strictly positive codimension for
 any $n$.
 Moreover, if such a sequence $\varepsilon_n$ converges to 0 at rate $k$ (so that $\lim_{n\to\infty} n^k\varepsilon_n=c$ for some constant $c$), we say that the
 measure concentrates on the family $\{S_n\}$ at least at rate $k$.
\end{tdef}

Notice that in general we may have $\mu_n(S_n)=0$.
Moreover, with this definitions we do not need any notion of convergence of $S_n$ to a final subset. $S_n$ just gives a ``direction of concentration''. Also, our definitions do not pretend to provide any optimality in concentration: it can happen that for
a given sequence of $S_n$ it exists a sequence of proper subsets $S'_n\subset S_n$ on which we still have concentration.
A special example of Definition \ref{def2} consists in the case $X_n=X$ and $S_n=S$ where $X$ is a metric space, $\{\mu_n\}_{n\in\mathbb N}$ a sequence of normalized measures on $X$, compatible with the metric of $X$ and $S\subset X$ a proper closed subset of $X$.


\subsection{The Macdonalds formula}\label{app:Macdonalds}
Let us consider any simple compact Lie algebra of dimension $d$ and rank $r$. It is characterised by $p=(d-r)/2$ positive roots $\alpha_i$, $i=1,\ldots,p$ from which one can pick out a fundamental set of simple roots, say $\alpha_i$, $i=1,\ldots,r$.
To each non vanishing root $\alpha_i$ is associated a coroot
\begin{align}
 \check \alpha=\frac {2\alpha}{(\alpha|\alpha)},
\end{align}
being $(|)$ the scalar product, induced by the Killing form, on the real form $H_{\mathbb R}^*$ of the dual $H^*$ of the Cartan subalgebra $H$. The simple coroots define a lattice whose fundamental cell represents the fundamental torus $T^r$.
The polynomial invariants of the algebra (and groups) are generated by $r$ fundamental invariants of degree $d_i$, $i=1,\ldots,r$, depending on the algebra.\\
For such a Lie algebra there can be several compact Lie groups having it as Lie algebra. All of them are obtained by taking the quotient of the unique compact simply connected group $G$ w.r.t. a subgroup $\Gamma$ of the center $Z$ of $G$:
$G_\Gamma=G/\Gamma$. $\Gamma$ is isomorphic to $\pi_1(G)$.
\begin{teo}[Hopf]
 The cohomology of a connected compact Lie group G of rank $r$ over a field of characteristic 0 is that of a product of $r$ odd dimensional spheres.
\end{teo}
See \cite{HoMo}.
Indeed, such spheres have dimension $D_i=2d_i-1$, $i=1,\ldots,r$, where $d_i$ are de degrees of the fundamental invariants. The Killing form induces on a simple Lie group a unique (up to normalisation) bi-invariant metric that gives to the compact
groups a Riemannian structure. In particular, the corresponding Riemannian volume form gives the Haar measure on the group. Normalising the metric by fixing the length of any given simple root completely fixes, by rigidity, the entire volume of the group,
which can then be computed by means of the Macdonald's formula \cite{Ma}, \cite{BeCaCe}, \cite{CaCeDVOrSc}:
\begin{align}
 V(G_\Gamma)=\frac 1{|\Gamma|}V(T^r)\prod_{i=1}^r V(S^{2d_i-1}) \prod_{i=1}^p (\check \alpha_i|\check \alpha_i), \label{Macdonalds}
\end{align}
where $|\Gamma|$ is the cardinality of $\Gamma$,
\begin{align}
 V(T^r)=|\check \alpha_1\wedge \ldots \wedge \check \alpha_r|
\end{align}
and
\begin{align}
 V(S^{2d_i-1})=2\frac {\pi^{d_i}}{(d_i-1)!}.
\end{align}

\section{Levy property from the Ricci Tensor}\label{app:Ricci Tensor}

We can change the property of being Levy or not, simply by rescaling the distances by $i$-dependent constants. In particular, if $X_i$, or better $(X_i,g_i)$, are compact Riemannian manifolds, if $\mu_{g_i}$ is the measure naturally
associated to $g_i$, we can then consider the family
\begin{align}
 Y_i=(X_i,g_i, \mu_i), \qquad\ \mu_i=\frac {\mu_{g_i}}{\mu_{g_i}(X_i)},
\end{align}
and ask whether it is Levy or not. A simple answer is given by Corollary of the Theorem 4.1 \cite{GrMi}: let $Ric_i$ the Ricci tensor determined by $g_i$ and define
\begin{align}
 R_i=\inf Ric_i(\tau,\tau)
\end{align}
taken in the set of all tangent vectors of unit length. The theorem states that if
\begin{align}
 \lim_{i\to\infty} R_i=+\infty,
\end{align}
then $Y_i$ is Levy.\\
We will now compute the Ricci tensor for the simple groups in order to prove that the classical sequences of simple Lie groups are Levy.
The Maurer-Cartan (Lie algebra valued) 1-form $\pmb j$ over a compact Lie group $G$ is related to the bi-invariant metric $\pmb g$ over $G$ by
\begin{align}
 \pmb g=-\kappa^2 K(\pmb j \otimes \pmb j),
\end{align}
where $\kappa$ is a real normalization constant (for example, chosen so that $G$ has volume $1$), and $K$ the Killing form over $Lie(G)$, which is negative definite since $G$ is compact. $\pmb j$ does satisfy the Maurer-Cartan equation
\begin{align}
 d\pmb j+\frac 12 [\pmb j, \pmb j]=0,
\end{align}
where $[,]$ is the Lie product combined with the wedge product, as usual. If we fix a basis $T_i$, $i=1,\ldots,d$, for $\mathfrak g=Lie(G)$, and define the structure constants by
\begin{align}
 [T_i,T_j]=\sum_{k=1}^d c_{ij}^{\ \ k} T_k,
\end{align}
we can set
\begin{align}
 \pmb j=\sum_{j=1}^d j^j T_j
\end{align}
and the Maurer-Cartan equation becomes
\begin{align}
 dj^k+\frac 12 \sum_{i,j} j^i\wedge j^j c_{ij}^{\ \ k}=0.
\end{align}
If we look at the components of $\pmb j$ as defining a vielbein $j^i$, $i=1,\ldots, d$, associated to a metric
\begin{align}
 \tilde g_{ij}=\delta_{ij} j^i\otimes j^j,
\end{align}
we see that the Maurer-Cartan equation can be seen as the structure equation for the Levi-Civita connection (in terms of the Ricci rotation coefficients):
\begin{align}
 dj^k+\sum_j\omega^k_{\ j} j^j=0,
\end{align}
which thus gives
\begin{align}
 \omega^k_{\ j}=\sum_i \frac 12 c_{ij}^{\ \ k} j^i.
\end{align}
The curvature two form is then
\begin{align}
 \Omega^k_{\ j}=d\omega^k_{\ j} +\sum_l \omega^k_{\ l}\wedge \omega^l_{\ j}.
\end{align}
Its components $R^k_{\ jlm}$ with respect to the vielbein are thus
\begin{align}
 R^k_{\ jlm}=\frac 14 \sum_s C_{lm}^{\ \ s} C_{js}^{\ \ k}
\end{align}
from which we see that the Ricci tensor is
\begin{align}\label{key}
 Ric_{ij}=-\frac 14 K_{ij},
\end{align}
where $K$ is the Killing form.
Let us fix the compact simple Lie group $G$ and fix any basis $\{T_i\}$ for the Lie algebra in the smallest faithful representation $\rho$. A standard choice is to assume the basis is orthonormalised w.r.t. the condition (standard normalisation,
see App. \ref{app:standard})
\begin{align}
- \frac 12 Tr(\rho(T_i)\circ \rho(T_j))=\delta_{ij} \label{standard}
\end{align}
which is natural since $G$ is compact. This is also a biinvariant metric, then, it exists a positive constant $\chi_{G}$ (independent from $\Gamma$) such that
\begin{align}
 K_{ij}=-\chi_G \delta_{ij}
\end{align}
so that
\begin{align}
 Ric_{ij}=\frac {\chi_G}4 \delta_{ij},
\end{align}
or in coordinates
\begin{align}
 Ric_{ij}=\frac {\chi_G}4 \tilde g_{ij}.\label{Ricci}
\end{align}
The coefficients $\chi_G$ for the classical series of simple groups are computed below. We have: $\chi_{SU(n)}=n+2$, $\chi_{SO(n)}=n-2$ and $\chi_{USp(2n)}=2n+2$. \\
Therefore, we get the following corollary of the Gromov-Milman theorem:
\begin{cor}\label{cor:1}
 Let
 \begin{align}
 Z_i=(G_i,\tilde g_i,\mu_i),
 \end{align}
 where $G_i$ is any one of the classical sequences of compact simple Lie group, considered in the previous section, $\tilde g_i$ the corresponding standardly normalised biinvarian metric, and $\mu_i$ the Riemannian normalised measure.
 Then $\{Z_i\}_i$ is a Levy family.
\end{cor}
 \begin{proof}
 From (\ref{Ricci}) we get
\begin{align}
 R_i=\frac {\chi_G}{4}.
\end{align}
From the values of $\chi_G$ we get
\begin{align}
 R_i=
\begin{cases}
 \frac {i+2}4 & {  \rm if  } \quad\ G=SU(i), \\
 \frac {i-2}4 & {  \rm if  } \quad\  G=SO(i), \\
 \frac {i+1}2 & {  \rm if  } \quad\  G=USp(2i).
\end{cases}
\end{align}
Then, $\lim_{i\to\infty} R_i=+\infty$.
 \end{proof}

 \subsection{Computation of $\chi_G$}
 The strategy for computing the coefficient $\chi_G$ is very simple: after choosing an orthonormal basis $T_i$ in the smallest faithful representation $\rho$, we use it to compute one of these matrices in the adjoint representation. Then
 \begin{align}
 \chi_g=-\frac 12 \rm {Tr} (ad_{T_1}^2). \label {normalisation}
 \end{align}
 We will indicate with $E_{i,j}$ the elementary matrix having as the only non vanishing element the one at line $i$ and column $j$, which is 1. \\
 {\bf The unitary case:} The $\rho$ representation of $\mathfrak {su}(n)$ is realised by the anti-hermitian $n\times n$ matrices having vanishing trace. A basis is given by $H_k,S_{kj},A_{kj}$, $k=1,\ldots,n-1$, $1\leq k<j\leq n$, where
 \begin{align}
 H_k&=\frac {i\sqrt 2}{\sqrt{k^2+k}} (E_{1,1}+\ldots+E_{k,k}-kE_{k+1,k+1}), \qquad\ k=1,\ldots,n-1, \\
 S_{k,j}&=i (E_{i,j}+E_{j,i}), \qquad\ k<j, \\
 A_{k,j}&= E_{k,j}-E_{j,k}, \qquad\ k<j.
 \end{align}
 Let us construct the adjoint matrix of $H_1$. The only non vanishing commutators of $H_1$ are
 \begin{align}
 [H_1,A_{1,2}]&=2 S_{1,2}, \qquad\qquad\ [H_1,S_{1,2}]=-2 A_{1,2},\\
 [H_1,A_{1,j}]&= S_{1,j}, \qquad\qquad\ [H_1,S_{1,j}]=- A_{1,2}, \qquad\ j=3,\ldots,n.
 \end{align}
In order to compute $(ad(H_1))^2$ we have to compute again the commutator, which gives
 \begin{align}
 ad_{H_1}^2(A_{1,2})&=- 4 A_{1,2}, \qquad\qquad\ ad_{H_1}^2(S_{1,2})=- 4 S_{1,2},\\
 ad_{H_1}^2(A_{1,j})&=- A_{1,j}, \qquad\qquad\ ad_{H_1}^2(S_{1,j})=- S_{1,j}, \qquad\ j=3,\ldots,n.
 \end{align}
 Taking the trace we get $\chi_{SU(n)}=n+2$.

 \

 \noindent {\bf The orthogonal case:} The $\rho$ representation of $\mathfrak {so}(n)$ is realised by the anti-symmetric $n\times n$ matrices. A basis is given by $A_{kj}$, $1\leq k<j\leq n$, where
  \begin{align}
 A_{k,j}&=E_{i,j}-E_{j,i}, \qquad\ k<j.
 \end{align}
 Let us consider $ad(A_{1,2})$. The only non vanishing commutators are
 \begin{align}
 [A_{1,2},A_{1,j}]=- A_{2,j}, \qquad\qquad\ [A_{1,2},A_{2,j}]= A_{1,j}, \quad\ j=3,\ldots, n.
 \end{align}
 Iterating the commutators, we get
 \begin{align}
 ad_{A_{1,2}}(A_{1,j})=- A_{1,j}, \qquad\qquad\ ad_{A_{1,2}}(A_{2,j})=- A_{2,j}, \quad\ j=3,\ldots, n.
 \end{align}
After taking the trace we get $\chi_{SO(n)}=n-2$.

 \

 \noindent {\bf The symplectic case:} The $\rho$ representation of $\mathfrak {usp}(n)$ is realised by the anti-hermitian $2n\times 2n$ matrices having the form
 \begin{align}
 \begin{pmatrix}
 A & B \\ C & -A^t
 \end{pmatrix},
 \end{align}
 where $B$ and $C$ are symmetric. A basis is given by
 \begin{align}
 H_a &=i(E_{a,a}-E_{a+n,a+n}), \qquad\ a=1,\ldots,n;\\
 S^d_{ij}&=\frac i{\sqrt 2} (E_{i,j}+E_{j,i}-E_{i+n,j+n}-E_{j+n,i+n}), \qquad\ i<j;\\
 A^d_{ij}&=\frac 1{\sqrt 2} (E_{i,j}-E_{j,i}+E_{i+n,j+n}-E_{j+n,i+n}), \qquad\ i<j;\\
 T_a&=i (E_{a,a+n}+E_{a+n,a}), \qquad\ a=1,\ldots,n;\\
 S^a_{ij}&=\frac i{\sqrt 2} (E_{i,j+n}+E_{j,i+n}+E_{i+n,j}+E_{j+n,i}), \qquad\ i<j;\\
 U_a&= (E_{a,a+n}-E_{a+n,a}), \qquad\ a=1,\ldots,n;\\
 A^a_{ij}&=\frac 1{\sqrt 2} (E_{i,j+n}+E_{j,i+n}-E_{i+n,j}-E_{j+n,i}), \qquad\ i<j.
 \end{align}
We consider the adjoint representation of $H_1$. The non vanishing commutators are
\begin{align}
 [H_1,S^d_{1,j}]&=- A^d_{1j}, \qquad\qquad\ [H_1,A^d_{1,j}]= S^d_{1j}, \quad\ j=2,\ldots,n, \\
 [H_1,T_1]&=-{2} U_1, \qquad\qquad\ [H_1,U_1]={2} T_1, \\
 [H_1,S^a_{1,j}]&=-A^a_{1j}, \qquad\qquad\ [H_1,A^a_{1,j}]= S^a_{1j}, \quad\ j=2,\ldots,n.
\end{align}
Iterating the commutators we get
\begin{align}
 ad_{H_1}^2(S^d_{1,j})&=-S^d_{1j}, \qquad\qquad\ ad_{H_1}^2(A^d_{1,j})=- A^d_{1j}, \quad\ j=2,\ldots,n, \\
 ad_{H_1}^2(T_1)&=-{4} T_1, \qquad\qquad\ ad_{H_1}^2(U_1)=-4 U_1, \\
 ad_{H_1}^2(S^a_{1,j})&=- S^a_{1j}, \qquad\qquad\ ad_{H_1}^2(A^a_{1,j})=- A^a_{1j}, \quad\ j=2,\ldots,n.
\end{align}
Finally, by taking the trace we get $\chi_{USp(2n)}=2n+2$.

\subsection{On the standard normalisation}\label{app:standard}
The standard normalisation of the metric has a clear meaning if referred to  the two-plane rotations, which are the rotations leaving fixed a codimension $2$ space. These are contained in each group, and are, for example,
the one generated by each one of the generators $A_{k,j}$ of $SU(n)$, each one of the generators of $SO(n)$, or each one of the $U_a$ in the symplectic case. In order to understand its meaning let us fix for example $A_{k,j}$ and
consider the one parameter subgroup defined by
\begin{align}
R\equiv  R(\theta)\equiv R_{k,j}(\theta) ={\exp} (\theta A_{k,j}).
\end{align}
It represents rotations of the $k-j$ plane by $\theta$, and has periodicity $2\pi$. Let us consider the normalised metric restricted to that orbit $O\equiv O_{jk}=R([0,2\pi])$. A simple calculation gives
\begin{align}
 g|_{O}=-\frac 12 {\rm Tr} (R^{-1}dR \otimes R^{-1}dR)=d\theta^2.
\end{align}
Thus, the total length of the whole orbit, correspondent to a continuous rotation up to a round angle, is exactly $2\pi$.

\section{Concentration Locus on compact Lie groups}\label{ConcLocLieGroup}
We will start by considering the concentration of measure on compact Lie group families by direct inspection of their geometries and of the invariant measures on them. Let us consider the cases of the classical series. In this case we will prove not only that one gets Levy families, but we will also individuate a concentration locus.


\subsection{Concentration Locus on simple compact Lie Groups}
We consider the classical series of simple Lie groups. We will always mean the simply connected form of the groups and will consider the standard normalization for the matrices. By Corollary \ref{cor:1}, any sequence of them is a levy family. In this section we make a concrete calculation of a concentration locus for each of them.


\subsubsection{Special unitary groups}\label{sec:unitary}
The group $SU(n)$ of unitary $n\times n$ matrices with unitary determinant is a simply connected group of rank $n-1$ and its Lie algebra is the compact form of $A_{n-1}$, that is $\mathfrak {su}(n)$.
The center is $\mathbb Z_n$, generated by the $n$-th roots of 1.
The fundamental invariant degrees
are $d_i=i+1$, $i=1,\ldots,n-1$. The spheres generating the cohomology have dimension $D_i=2i+1$. With the standard normalisation a fundamental system of simple root can be represented as follows:\\
one identifies isometrically $H^*_{\mathbb R}$ with an hyperspace of $\mathbb R^{n}$, as
\begin{align}
 H^*_{\mathbb R} \simeq \{(x_1,\ldots,x_n)\in \mathbb R^n| x_1+\ldots+x_n=0\}.
\end{align}
In this representation, if $\pmb e_i$, $i=1,\ldots,n$ is the canonical (orthonormal) basis of $\mathbb R^n$, the simple roots are
\begin{align}
 \alpha_i=\pmb e_i-\pmb e_{i+1}, \qquad\ i=1,\ldots, n-1.
\end{align}
All roots have square length $2$, and coincide with the coroots. The dimension of the group is $n^2-1$, so that there are $p=n(n-1)/2$ positive coroots. The volume of the torus is
\begin{align}
 V(T^{n-1})=|(\pmb e_1-\pmb e_2)\wedge \ldots \wedge (\pmb e_{n-1} -\pmb e_n)|=\sqrt n.
\end{align}
Thus, the Macdonalds formula (\ref{Macdonalds}) gives
\begin{align}
 V(SU(n))=\frac {\sqrt n (2\pi)^{\frac {n(n+1)}2-1}}{\prod_{i=1}^{n-1}i!}.
\end{align}
It follows that
\begin{align}
 \frac {V(SU(n+1))}{V(SU(n))}=\sqrt{\frac {n+1}{n}} \frac {(2\pi)^{n+1}}{n!}\sim \sqrt {\frac {2\pi}n} \left(\frac {2\pi e}n\right)^n,
\end{align}
so that, since dim$SU(n+1)-$dim$SU(n)=2n+1$, we have
\begin{align}
\left( \frac {V(SU(n+1))}{V(SU(n))}\right)^{\frac 1{2n+1}}\sim \left(\frac {2\pi e}n\right)^{1/2}.\label{rapportosun}
\end{align}
This is substantially the same behaviour as for the spheres (of radius 1), \cite{Le}, and it is enough to prove the concentration of the measure. Indeed, it means that the volume of $SU(n)\subset SU(n+1)$ grows mush faster with $n$ than the volume of
$SU(n)$. This means that if we take the normal bundle of $SU(n)$ in $SU(n+1)$ and take a neighbourhood $\mathcal T_n$ of $SU(n)$ of radius $\varepsilon$ in the normal directions, we get for the volume of this neighbourhood
\begin{align}
 \frac {V(SU(n+1))}{\mathcal T_n} \sim \sqrt {\frac {2\pi}n} \left(\frac {2\pi e}{n\varepsilon^2}\right)^n \frac 1{\varepsilon}.
\end{align}
which for any given $\varepsilon$ decreases to $0$ when $n\to \infty$.
However, it does not give us direct information on how the concentration sets move.
A more precise result is the following.
\begin{prop}\label{ConLocus1}
Consider the family of simple Lie groups $SU(n+1)$ endowed with the usual biinvariant metric.
 Let us consider the Hopf structure of $SU(n+1)$, t.i. $U(n) \hookrightarrow SU(n+1) \longrightarrow \mathbb {CP}^n$. Let $S_n$ be the hyperplane at infinity in $\mathbb {CP}^n$, and
\begin{align}
 \iota : S_n \hookrightarrow \mathbb {CP}^n
\end{align}
the corresponding embedding. Finally, let $\mu_n$ the normalised invariant measure on $SU(n+1)$. Then, after looking at $SU(n+1)$ as a $U(n)$-fibration over $\mathbb {CP}^n$, the invariant measure concentrates on
the real codimension 2 subvariety
\begin{align}
\Sigma_n= \iota^* (SU(n+1)),
\end{align}
in the sense of definition \ref{def2}, with constant $\varepsilon$.
\end{prop}
\begin{proof}
Recall that $U(n)\subset SU(n+1)$ is a maximal proper Lie subgroup and $\mathbb {CP}^n=SU(n+1)/U(n)$ (and $SU(n)\subset U(n)$). Therefore, one expects for the measure $\mu_{SU(n+1)}$ to factorise as
\begin{align}
 d\mu_{n}=d\mu_{\mathbb {CP}^n} \times d\mu_{U(n)}.
\end{align}

Now $\mathbb {CP}^n\simeq S^{2n+1}/U(1)$ and the natural metric over it is the Fubini-Study metric that is invariant under the action of the whole $SU(n+1)$ group. Thus, we expect the measure $d\mu_{\mathbb {CP}^n}$, inherited from the
whole invariant measure, to be the Riemannian volume form corresponding to the Fubini-Study metric. On the other hand, the relation between $\mathbb {CP}^n$ and $S^{2n+1}$ suggests that the concentration of the measure of $d\mu_{\mathbb {CP}^n}$ should happen
over some codimension two submanifold $S\subset \mathbb {CP}^n$. This would imply that the whole invariant measure of $SU(n+1)$ concentrates on a $U(n)$ fibration over $S$. This is the strategy of the proof that we will now explicit out.
To this aim, we employ the explicit construction of the invariant measure over Lie groups given in \cite{CaDaPiSc}. In particular, the analysis of the geometry underlying the construction of the invariant measure
for $SU(n)$ has been performed in \cite{BeCaCe}.
Fix a generalized Gell-Mann basis $\{\lambda_I\}_{I=1}^{n^2+2n}$ for the Lie algebra of $SU(n+1)$ as in \cite{BeCaCe}. Thus, the first $n^2$ matrices generate the maximal subgroup $U(n)$, the last one being the $U(1)$ factor, and,
in particular, the matrices $\{\lambda_{(a+1)^2-1}\}_{a=1}^{n}$ generate the Cartan torus $T^n$. Then, the parametrization of $SU(n+1)$ can be obtained inductively as
\begin{align}
 SU(n+1)\ni g=h\cdot u,
\end{align}
where $u\in U(n)$ is a parametrization of the maximal subgroup, and
\begin{align}
 h=e^{i\theta_1\lambda_3} e^{i\phi_1 \lambda_2} \prod_{a=2}^n [e^{i(\theta_a/\epsilon_a)\lambda_{a^2-1}}e^{i\phi_a\lambda_{a^2+1}}], \qquad\ \epsilon_a=\sqrt {\frac 2{a(a-1)}},
\end{align}
parametrizes the quotient. From $h$ one can construct a vielbein for the quotient as follows. Let $J_h$ be the Maurer-Cartan 1-form of $SU(n+1)$ restricted to $h$. Then set
\begin{align}
 e^l=\frac 12 {\rm Tr} [j_h\cdot \lambda_{n^2+l-1}], \quad\ l=1,\ldots,2n.
\end{align}
They form a vielbein for $SU(n+1)/U(n)\simeq \mathbb {CP}^n$ so that
\begin{align}
 & ds^2_{\mathbb {CP}^n} =\delta_{lm} e^l\otimes e^m,\label{metric}\\
 & d\mu_{\mathbb {CP}^n}=\det \underline e
\end{align}
are the metric and invariant measure respectively, induced on $\mathbb {CP}^n$. In particular, one gets
\begin{align}
 \det \underline e =2d\theta_n d\phi_n \cos \phi_n \sin^{2n-1} \phi_n \prod_{a=1}^{n-1} [\sin \phi_a \cos^{2a-1} \phi_a d\theta_a d\phi_a]. \label{measure}
\end{align}
One can also write down the metric. Indeed, it has been shown in \cite{BeCaCe} that it is exactly the Fubini-Study metric for $\mathbb {CP}^n$ written in unusual coordinates. Since this is relevant for our analysis, let us summarise it. Let
$(\zeta_0:\cdots:\zeta_n)$ be the homogeneous coordinates and
\begin{align}
 \mathcal K=\frac 12 \log (|\zeta_0|^2+\ldots+|\zeta_n|^2)
\end{align}
be the K\"ahler potential. Fix a coordinate patch, say $U_0=\{ \underline \zeta: \zeta_0\neq 0 \}$ with the relative non-homogeneous coordinates $z_i=\zeta_i/\zeta_0$, $i=1,\ldots,n$. When $\underline z$ varies in $\mathbb C^n$, the coordinate
patch covers the whole $\mathbb {CP}^n$ with the exception of a real codimension two submanifolds defined by the hyperplane
\begin{align}
S_n\equiv \mathbb {CP}^{n-1} =\{ 0:\zeta_1 :\cdots: \zeta_n\},
\end{align}
the so called hyperplane at infinity.
In these local coordinates the Fubini-Study metric has components $g_{i\bar j}=\partial^2 \mathcal K/\partial z_i \partial \bar z_{j}$:
\begin{align}
 ds^2_{F-S}=\frac {\sum_i dz_i d\bar z_i}{1+\sum_j |z_j|^2}- \frac {\sum_{i,j} \bar z_i dz_i z_j d\bar z_j}{(1+\sum_j |z_j|^2)^2}.
\end{align}
Following \cite{BeCaCe}, let us introduce the change of coordinates
\begin{align}
 z_i=\tan \xi R_i(\underline \omega) e^{i\psi_i}
\end{align}
where $R_j(\underline \omega)$ is an arbitrary coordinatization of the unit sphere $S^{n-1}$, $\psi_i\in [0,2\pi)$, $\xi\in[0,\pi/2)$.
In these coordinates
\begin{align}
 ds^2_{F-S}=d\xi^2 +\sin^2 \xi \left[ \sum_i dR_idR_i +\sum_i  R_i^2 d\psi_i d\psi_i \right]-\sin^4 \xi \left[ \sum_i R_i^2 d\psi_i \right]^2.\label{metrica-FS-mod}
\end{align}
In \cite{BeCaCe} it has been proved that this metric coincides with \eqref{metric}, after a simple change of variables, which, in particular, includes $\xi=\phi_n$. On the other hand, from \eqref{measure}, using
\begin{align}
 \int_0^{\pi/2-\varepsilon} \cos \phi_n \sin^{2n-1} \phi_n d\phi_n= \frac {\cos^{2n}\varepsilon}{2n}, \label{zero}
\end{align}
we see that the measure over $\mathbb {CP}^{n}$ concentrates around $\phi_n=\xi=\pi/2$. Finally, since
\begin{align}
 (1:\tan \xi R_1(\underline \omega ) e^{i\psi_1}:\cdots :\tan \xi R_n(\underline \omega ) e^{i\psi_n})&=(1/\tan \xi : R_1(\underline \omega ) e^{i\psi_1}:\cdots :R_n(\underline \omega ) e^{i\psi_n}) \cr & \mapsto
 (0:R_1(\underline \omega ) e^{i\psi_1}:\cdots :R_n(\underline \omega ) e^{i\psi_n})
\end{align}
when $\xi\to\pi/2$, we see that the concentration is on the hyperplane $S_n$ at infinity. Thus , if
\begin{align}
 \iota : S_n \hookrightarrow \mathbb {CP}^n
\end{align}
is the embedding of the hyperplane and if we look at $SU(n+1)$ as a fibration over $\mathbb {CP}^n$, we get that the whole measure concentrates on
\begin{align}
\Sigma= \iota^* (SU(n+1)),
\end{align}
which is what we had to prove.
\end{proof}
\begin{rem}\rm
It is worth to remark that we are not saying the the sequence of manifolds we have selected completely describe the concentration. Indeed, it is obvious that the concentration can take place on proper subspaces of the sequence. For example,
(\ref{zero}) shows that the volume of the region $B_\varepsilon$ given by $|\xi|> \varepsilon$ from the concentration locus has volume vanishing as
\begin{align}
 \sim e^{-n\varepsilon^2}
\end{align}
If we now take $\varepsilon \to \varepsilon/\sqrt N$ and consider $N$ regions $B^k_{\varepsilon/N}$, $k=1,\ldots,N$ associated to $N$ $\mathbb {CP}_n$ planes intersecting transversally in a point $p$ of our concentration locus. Then,
$\bigcup_{k=1}^N B^k_{\varepsilon/N}$ has volume of order $V_n \sim Ne^{-n\varepsilon^2/N}$. Its complement is a subset of codimension $N$ of the concentration locus. In order to have $V_n\to0$ it is sufficient that, for example,
\begin{align}
 Ne^{-n\varepsilon^2/N}\sim \frac Nn
\end{align}
to that we can consider $N\equiv N_n$ as dependent on $n$, with the condition that
\begin{align}
 N_n \frac {\log n}n \to 0
\end{align}
when $n\to\infty$ with $\varepsilon$ fixed. Therefore, the codimension in general can diverge, and we get concentration loci of divergent codimension. This shows that it is not clear at all if a notion of optimal concentration can be defined.
\end{rem}
\begin{rem} \rm
In order to get uniform concentration in the sense of Gromov and Milman, we have to add a further hypothesis to our proposition, already suggested by formula (\ref{rapportosun}). From Corollary \ref{cor:1}, we see that if one normalizes the size of
$SU(n+1)$ so that its coroots have length 2 (the standard choice), then its scalar curvature is $R_n=\frac {n+3}4$. However, we can, in general, relax this condition and leave the length $|\check \alpha|$ of the coroots free. In this case, the scalar curvature
becomes
\begin{align}
 R_n= \frac {n+2}{|\check \alpha_n|^2}.
\end{align}
Following \cite{GrMi}, we see that we have a Levy family if $|\check \alpha_n|$ grows less than $\sqrt n$. Finally, it is interesting to notice that if we approximate the shape of the group as the product of $n$ spheres of radius $|\check \alpha_n|$, then its
diameter scales as $|\check \alpha_n| \sqrt n$. Thus the uniform concentration is guaranteed if the diameter of the group grows less than $\sim n$. Finally, since the dimension of $SU(n+1)$ is $d_n=n^2+2n$, we see that the condition is such that the diameter
must grow less than $\sqrt {d_n}$, which is very similar to the case of the spheres.\\
Finally, this can also be understood from (\ref{zero}) also. Indeed, keeping the diameter fixed, we see that the $\varepsilon$ dependence is dominated by the therm $\cos^{2n}\varepsilon=(1-\sin^2\varepsilon)^n$. In place of rescaling the diameter, assume
we rescale $\varepsilon$ in a $n$-dependent way, so $\varepsilon\to \varepsilon_n$, and assume that $\varepsilon_n\to 0$ when $n\to\infty$. Therefore, for large $n$ we have
\begin{align}
 \cos^{2n}\varepsilon\sim e^{-n\varepsilon_n^2},
\end{align}
which converges to zero only if $n\varepsilon_n^2\to\infty$. This means that $\varepsilon_n$ must decrease to $0$ slower than $n^{-\frac 12}$, which is as to say that $\varepsilon/|\check \alpha|$ must go to zero slower than $n^{-\frac 12}$ independently
from how we allow $\varepsilon$ and $|\check \alpha |$ to vary separately with $n$.	
\end{rem}

\begin{rem} \rm
The concentration metric in the form (\ref{metrica-FS-mod}) becomes degenerate at the concentration locus when $\xi=\frac \pi2$, since one has to further fix one of the phases $\psi_j$. Nevertheless, if we consider
a region $V_r$ of $\xi$-radius $\pi/2-\xi=r$ around that locus, since the total measure is normalized to 1, we see from (\ref{zero}) that its volume is
\begin{align}
 \mu(V_r)=1-\cos^{2n} r.
\end{align}
As in \cite{ledoux}, chapter 2.1, we can use the inequality $\cos r\leq e^{-\frac {r^2}2}$ for $0\leq r\leq \frac \pi2$, so that
\begin{align}
 \mu(V_r)\geq 1-e^{-nr^2},
\end{align}
which gives us an estimation of how much the measure concentrates around the singular locus: for any fixed $r>0$, the measure of $V_r$ converges exponentially to the full measure when $n$ increases.\\
\end{rem}

It is worth mentioning that the limit topology depends not only on the topology of each space of the chain but also from the embeddings defining the sequence of groups.
For example, we can replace the canonical embedding $U(n)\subset U(n+1)$ with the embeddings
\begin{align}
 U(n) \stackrel J{\hookrightarrow} SU(n+1) \subset U(n+1)
\end{align}
with
\begin{align*}
 J(X)=
\begin{pmatrix}
 X & \vec 0 \\ \vec 0^t & \det X^{-1}
\end{pmatrix}.
\end{align*}
These embeddings lead to the result $SU(\infty)_J=U(\infty)_J$ for any limit topology we choose. Observe that if you use canonical embeddings it is unknown whether the inductive limit $SU(\infty)$ is extremely amenable or not \cite{Pe2}.

\subsubsection{Odd special orthogonal groups}
The second classical series of simple groups is given by the odd dimensional special orthogonal groups $SO(2n+1)$ of dimension $n(2n+1)$ and rank $n$. The center of the universal covering $Spin(2n+1)$ is $\mathbb Z_2$. The Lie algebra is the
compact form of $B_n$, $n\geq 2$.
The invariant degrees are $d_i=2i$, $i=1,\ldots,n$ and the dimensions of the spheres generating the cohomology are $D_i=4i-1$. If we choose the standard normalisation, a fundamental system of simple roots in
$\mathbb R^n\simeq H^*_{\mathbb R}$ is given by $\alpha_i=\pmb e_i-\pmb e_{i+1}$, $i=1,\ldots,n-1$, and $\alpha_n=\pmb e_n$. The corresponding coroots are $\check \alpha_i=\alpha_i$ for $i=1,\ldots,n-1$, and $\check \alpha_n=2\alpha_n$.
There are $p=n^2$ positive coroots, $n$ of which have length 2 and the others have square length 2. The volume of the torus is
\begin{align}
 V(T^n)=|(\pmb e_1-\pmb e_2)\wedge (\pmb e_{n-1}-\pmb e_n)\wedge 2\pmb e_n|=2.
\end{align}
The Macdonald's formula thus gives
\begin{align}
 V(Spin(2n+1))=\frac {2^{n(n+2)+1}\pi^{n(n+1)}}{\prod_{i=1}^n (2i-1)!},
\end{align}
so that
\begin{align}
 \frac {V(Spin(2n+1))}{V(Spin(2n-1))}=\frac {2^{2n+1}\pi^{2n}}{(2n-1)!}\sim \sqrt {\frac {4\pi}{n-\frac 12}} \left(\frac {2\pi e}{2n-1}\right)^{2n-1}.
\end{align}
Since dim$Spin(2n+1)-$dim$Spin(2n-1)=4n-1$, we have
\begin{align}
\left( \frac {V(Spin(2n+1))}{V(Spin(2n-1))}\right)^{\frac 1{4n-1}}\sim \left(\frac {2\pi e}{2n}\right)^{1/2},
\end{align}
which shows the same behaviour as for the unitary groups. Again, in order to understand how concentration works, we have to do some geometry.
\begin{prop}\label{ConcLocus2}
Consider the sequence of simple groups $Spin(2n+1)$ endowed with the biinvariant metric.
 Set $B_n=S^{2n}\times S^{2n-1}\equiv Spin(2n+1)/Spin(2n-1)$ so that $Spin(2n+1)$ looks as a $Spin(2n-1)$-fibration over $B_n$. Finally, let $S_n$ a bi-equator of $B_n$ (the cartesian product of the equators of the two spheres), and
 \begin{align}
 \iota : S_n \hookrightarrow B_n
\end{align}
the corresponding embedding. Then, in the limit $n\to\infty$ the invariant measure $\mu_n$ of $Spin(2n+1)$ concentrates on the codimension two subvariety
\begin{align}
\Sigma_n= \iota^* (Spin(2n+1)),
\end{align}
in the sense of definition \ref{def2}.
\end{prop}
\begin{proof}
Since the proof is much simpler than in the previous case, we just sketch it, leaving the details to the reader.
By using the methods in \cite{CaDaPiSc}, in a similar way as before, it is easy to prove that the invariant measure $d\mu_n$ factorises as
\begin{align}
 d\mu_{Spin(2n+1)}=d\mu_{Spin(2n-1)}\times dm_{S^{2n}} \times dm_{S^{2n-1}},
\end{align}
where $dm$ is the Lebesgue measure. Therefore, since is well known that the Lebesgue measures on the spheres concentrate over the equators
we get again that the measure $d\mu_n$ concentrates on a $Spin(2n-1)$ fibration over a codimension two submanifold of $S^{2n}\times S^{2n-1}$.
\end{proof}

\vspace{0.5cm}



\subsubsection{Symplectic groups}
The compact form $USp(2n)$ of the symplectic group of rank $n$ has dimension $2n^2+2$. Its center is $\mathbb Z_2$ and its Lie algebra is the compact form of $C_n$, $n\geq 2$. The invariant degrees are the same as for $SO(2n+1)$, so
they have the same sphere decomposition. In the standard normalisation the roots of $USp(2n)$ are the coroots of $SO(2n+1)$ and viceversa. Therefore, we have $n^2-n$ coroots of length $\sqrt 2$ and $n$ of length 1. The volume
of the torus is
\begin{align}
 V(T^n)=|(\pmb e_1-\pmb e_2)\wedge (\pmb e_{n-1}-\pmb e_n)\wedge \pmb e_n|=1,
\end{align}
and the volume of the group is
\begin{align}
 V(Usp(2n))=\frac {2^{n^2}\pi^{n(n+1)}}{\prod_{i=1}^n (2i-1)!}.
\end{align}
Again, we get
\begin{align}
\left( \frac {V(USp(2n))}{V(USp(2n-2))}\right)^{\frac 1{4n-1}}\sim \left(\frac {2\pi e}{2n}\right)^{1/2}.
\end{align}
\begin{prop}
Consider the sequence of symplectic groups $USp(2n)$ endowed with the biinvariant metric.
 Set $S^{4n-1}\equiv USp(2n)/USp(2n-2)$ so that $USp(2n)$ looks as an $USp(2n-2)$-fibration over $B_n=S^{4n-1}$. Finally, let $S_n$ an equator of $B_n$, and
 \begin{align}
 \iota : S_n \hookrightarrow B_n
\end{align}
the corresponding embedding. Then, in the limit $n\to\infty$ the invariant measure $\mu_n$ of $Spin(2n+1)$ concentrates on the codimension one subvariety
\begin{align}
\Sigma_n= \iota^* (USp(2n)),
\end{align}
in the sense of definition \ref{def2}.
\end{prop}
\noindent The proof is the same as for the spin groups.



\subsubsection{Even special orthogonal groups}
The last series is given by the even dimensional special orthogonal groups $SO(2n)$ of dimension $n(2n-1)$ and rank $n$. The center of the universal covering $Spin(2n)$ is $\mathbb Z_2\times \mathbb Z_2$ if $n=2k$, and $\mathbb Z_4$ if $n=2k+1$.
The Lie algebra is the compact form of $D_n$, $n\geq 4$.
The invariant degrees are $d_i=2i$, $i=1,\ldots,n-1$, $d_n=n$ and the dimensions of the spheres generating the cohomology are $D_i=4i-1$, $i=1,\ldots,n-1$, $D_n=2n-1$. If we choose the standard normalisation, a fundamental system of simple roots in
$\mathbb R^n\simeq H^*_{\mathbb R}$ is given by $\alpha_i=\pmb e_i-\pmb e_{i+1}$, $i=1,\ldots,n-1$, and $\alpha_n=\pmb e_{n_1}+\pmb e_n$.
The corresponding coroots are $\check \alpha_i=\alpha_i$ for $i=1,\ldots,n$, and all have length $\sqrt 2$.
There are $p=n^2-n$ positive coroots. The volume of the torus is
\begin{align}
 V(T^n)=|(\pmb e_1-\pmb e_2)\wedge (\pmb e_{n-1}-\pmb e_n)\wedge (\pmb e_{n-1}+\pmb e_n)|=2.
\end{align}
Thus,
\begin{align}
 V(Spin(2n))=\frac {2^{n^2+1}\pi^{n^2}}{(n-1)!\prod_{i=1}^{n-1} (2i-1)!},
\end{align}
and
\begin{align}
 \frac {V(Spin(2n))}{V(Spin(2n-2))}=\frac {2(2\pi)^{2n-1}}{(2n-2)!}\sim \sqrt {\frac {4\pi}{n-1}} \left(\frac {2\pi e}{2n-2}\right)^{2n-2}.
\end{align}
Since dim$Spin(2n)-$dim$Spin(2n-2)=4n-3$, we have
\begin{align}
\left( \frac {V(Spin(2n))}{V(Spin(2n-2))}\right)^{\frac 1{4n-3}}\sim \left(\frac {2\pi e}{2n}\right)^{1/2},
\end{align}
which, again, shows concentration.
\begin{prop}
Consider the sequence of simple groups $Spin(2n+2)$ endowed with the biinvariant metric.
 Set $B_n=S^{2n+1}\times S^{2n}\equiv Spin(2n+2)/Spin(2n)$ so that $Spin(2n+2)$ looks as a $Spin(2n)$-fibration over $B_n$. Finally, let $S_n$ a bi-equator of $B_n$, and
 \begin{align}
 \iota : S_n \hookrightarrow \mathbb B_n
\end{align}
the corresponding embedding. Then, in the limit $n\to\infty$ the invariant measure $\mu_n$ of $Spin(2n+2)$ concentrates on the codimension two subvariety
\begin{align}
\Sigma_n= \iota^* (Spin(2n+2))
\end{align}
in the sense of Definition~\ref{def2}.
\end{prop}


This exhausts the classical series. Further considerations can be made by using the Riemannian structure analysed in Sec. \ref{app:Ricci Tensor}. Here we limit ourselves to notice that in principle we can construct a huge number of
Levy families as a consequence of Theorem 1.2, page 844 of \cite{GrMi}:
\begin{cor}\label{cor:2}
 Let $Y_i=(X_i,g_i, \mu_i)$ a family of compact Riemannian spaces with natural normalised Riemannian measures. Assume there is a positive constant $c>0$ such that definitely $R_i\geq c$, where
\begin{align}
 R_i=\inf Ric_i(\tau,\tau)
\end{align}
taken in the set of all tangent vectors of unit length.
 Consider any sequence of positive constants $c_i$
 such that
 \begin{align}
 \lim_{i\to\infty} c_i=\infty.
 \end{align}
 Then, the new family
 \begin{align}
 \tilde Y_i=(X_i, \tilde g_i, \mu_i), \qquad\ \tilde g_i=\frac 1{c_i} g_i
 \end{align}
 is Levy.
\end{cor}
\begin{proof}
 Obviously $\tilde R_i=c_i R_i$. Since definitely $R_i\geq c$, we have $\lim_{i\to\infty} \tilde R_i=+\infty$.
\end{proof}
\section{Further comments and conclusions}
In a companion paper, \cite{CU21}, we have introduced the notion of ``concentration locus'' for sequences of groups $G_n$, $G_n\subseteq G_{n+1}$, endowed with normalized invariant measures. Then, we have shown in which sense the mapping of the
concentration locus on a set through its action on that set governs the concentration of the measure on the set and eventually determines the presence of a fixed point. Here we have seen how a concentration locus can be determined for the classical series
of compact Lie groups. This loci can have unboundedly increasing codimension and determine probes for analysing the action of some infinite dimensional Lie groups on (non necessarily) compact sets or manifolds. We remark that the question about
extreme amenability of $SU(\infty)$, for example, is still an open problem, \cite{Pe2}. The result we obtained for the classical series can be easily extended to more general sequences of compact Lie groups.
\begin{prop}
 Let $\{G_n\}$ a family of connected compact Lie groups of the form
\begin{align}\label{caso generale}
G_n= G_n^{(1)}\times \cdots \times G_n^{(k_n)}\times T^{s_n}/Z_{G_n},
\end{align}
where $T^{s_n}$ is a torus of dimension $s_n$, $G_n^{(1)}\times \cdots \times G_n^{(k_n)}$ is the product of $k_n$ compact connected simple Lie group and $Z_{G_n}$ is a finite subgroup.
Suppose that among the factor of $G_n^{(1)}\times \cdots \times G_n^{(k_n)}$ it exists a finite dimensional connected compact Lie group $G_0$ common to all $n$. Alternatively assume that $s_n\neq 0$ for $n>n_0$. Then, it exists at least
a finite dimensional compact manifold $K$ admitting an equicontinuous action of $G_\infty$, taken with the inductive limit topology, without fixed points. \\
If at least one of the $G_n^{j_n}$ determines a classical sequence $\{G_n^{j_n} \}_{n\in N}$ of compact Lie groups, then a concentration locus of $G_n$ is obtained restricting the factors $G_n^{j_n}$ to the corresponding concentration loci.
\end{prop}
The proof is simple and is left to the reader. In the first part obviously, $K=G_0$ or $K=S^1$. It generalizes the known result that $U(\infty)$ is not extremely amenable. The second part is just a corollary of our results in the previous sections.\\
It would be interesting to relate the concentration of the measure around concentration loci to the phenomenon of optimal transport. We expect such a connection to be governed by the way the process of concentration around concentration loci is
realized in our examples. We plan to investigate such connection in a future work.

\section*{Acknowledgenments}
The second author gratefully acknowledges partial support from the projects MEGABIT -- Universit\`{a} degli Studi di Catania, PIAno di inCEntivi per la RIcerca di Ateneo 2020/2022 (PIACERI), Linea di intervento 2.



\end{document}